\documentclass{amsart}
\usepackage{tikz}

\tikzstyle{vertex}=[auto=left,circle,fill=black!0,minimum size=0pt,inner sep=0pt]

\newtheorem{theorem}{Theorem}[section]
\newtheorem{lemma}[theorem]{Lemma}

\newtheorem{proposition}[theorem]{Proposition}
\theoremstyle{definition}  
\newtheorem{definition} [theorem] {Definition}

\newcommand{\Z}{{\mathbb{Z}}}

\newcommand{\Diam}{\mathrm{diam}}
\newcommand{\Link}{N_G}
\newcommand{\Linkp}{N_G^>}
\newcommand{\Linkm}{N_G^<}
\newcommand{\Linki}[1]{N_G^{\hskip.5pt\raisebox{.75pt}{\scriptsize$\star$}}(#1)}

\numberwithin{equation}{section}

\begin{document}

\title{On Closed Graphs II}

\author{David A.\ Cox}
\address{Department of Mathematics and Statistics, Amherst
College, Amherst, MA 01002-5000, USA}
\email{dacox@amherst.edu}

\author{Andrew Erskine}
\address{Department of Mathematics and Statistics, Amherst
College, Amherst, MA 01002-5000, USA}
\email{aperskine@gmail.com}

\begin{abstract}
A graph is closed when its vertices have a labeling by $[n]$ with a
certain property first discovered in the study of binomial edge
ideals.  In this article, we explore various aspects of closed graphs,
including the number of closed labelings and clustering coefficients.
 \end{abstract}


\maketitle

\section{Introduction}
\label{intro}

Given a simple graph $G$ with vertices $V(G)$ and edges $E(G)$, a
\emph{labeling} of $G$ is a bijection $V(G) \simeq [n] =
\{1,\dots,n\}$.  Given a labeling, we assume $V(G) = [n]$.

\begin{definition}
A labeling of $G$ is \emph{closed} when $\{j,i\}, \{i,k\} \in E(G)$
with $j > i < k$ or $j < i > k$ implies $\{j,k\} \in E(G)$.  Then $G$
is \emph{closed} if it has a closed labeling.
\end{definition}

A labeling of $G$ gives a direction to each edge $\{i,j\} \in E(G)$
where the arrow points from $i$ to $j$ when $i < j$, so the arrow
points to the bigger label.  In this context, closed means that
when two edges point away from a vertex or towards a vertex,
the remaining vertices are connected by an edge.

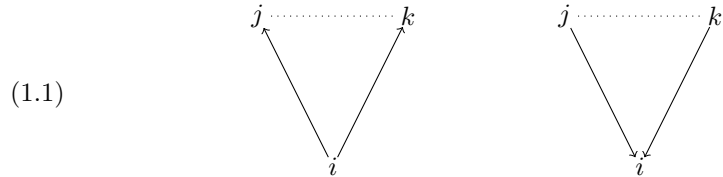
\begin{equation}
\label{closedpicture}
\begin{array}{ccc}
\begin{tikzpicture}
  \node[vertex] (n1)  at (2,1) {$i$};
  \node[vertex] (n2)  at (1,3)  {$j$};
  \node[vertex] (n3) at (3,3) {$k$};
 
  \foreach \from/\to in {n1/n2,n1/n3}
\draw[->] (\from)--(\to);;

 \foreach \from/\to in {n2/n3}
\draw[dotted] (\from)--(\to);;

\end{tikzpicture}&\hspace{30pt}&
\begin{tikzpicture}
  \node[vertex] (n1)  at (2,1) {$i$};
  \node[vertex] (n2)  at (1,3)  {$j$};
  \node[vertex] (n3) at (3,3) {$k$};
 
  \foreach \from/\to in {n2/n1,n3/n1}
\draw[->] (\from)--(\to);;

 \foreach \from/\to in {n2/n3}
\draw[dotted] (\from)--(\to);;

\end{tikzpicture}
\end{array}
\end{equation}

Closed graphs were first encountered in the study of binomial edge
ideals defined in \cite{H} and \cite{O}.  Properties of these ideals
are explored in \cite{E, S} and their relation to closed graphs
features in \cite{CR, E3,E4,E2}.

It is natural to ask for a characterization of those graphs that have
a closed labeling. One solution was given in \cite{CR}, which
characterizes closed graphs using the clique complex of $G$.  Another
approach, taken in our previous paper \cite{CE}, shows that a
connected graph is closed if and only if is chordal, claw-free, and
narrow (see \cite[Def.\ 1.3]{CE} for the definition of narrow).

In this paper, we will use tools developed in \cite{CE} to study the
combinatorial properties of closed graphs.  Our main results include:
\begin{itemize}
\item Section~\ref{collapsed}: Theorem~\ref{quotient}
  counts the number of closed labelings of a closed graph.
\item Section~\ref{counting}: Theorem~\ref{partition}
  counts the number of closed graphs with fixed layer structure (see
  Section~\ref{properties} for the definition of layer).
\item Section~\ref{clustering}: Theorem~\ref{CWSprop} gives a sharp
  lower bound for the clustering coefficient of a closed graph.
\end{itemize} 
To prepare for these results, we will recall some relevant results and
definitions in Section~\ref{properties} and explore when a labeling
remains closed after exchanging two labels in
Section~\ref{exchangeable}.

\section{Notation and Known Results}
\label{properties}

We recall some notation and results from \cite{CE}.  The
\emph{neighborhood} of $v \in V(G)$ is
\[
N_G(v) = \{ w \in V(G) \mid \{v,w\} \in E(G)\}.
\]
When $G$ is labeled and $i \in V(G) = [n]$, we have a disjoint union
\[
N_G(i) = N_G^>(i) \cup N_G^<(i),
\]
where
\[
N_G^>(i) = \{j \in N_G(i) \mid j > i\}\ \text{and}\ 
N_G^<(i) = \{j \in N_G(i) \mid j < i\}.
\]
Also, vertices $i,j \in [n]$ with $i \le j$ give the \emph{interval}
$[i,j] = \{k \in [n] \mid i \le k \le j\}$.

Here is a characterization of when a labeling of a connected graph is
closed.

\begin{proposition}[{[1, Prop.\ 2.4]}]
\label{nbdinterval}
A labeling on a connected graph $G$ is closed if and only if for all
$i \in [n]$, $N_G^>(i)$ is a complete subgraph and is an interval.
\end{proposition}

When a connected graph $G$ has a labeling with $V(G) = [n]$, we can
decompose $G$ into layers as follows. The \emph{$N^{\mathit{th}}$
layer of $G$} is the set $L_N$ of all vertices that are distance $N$
from vertex $1$, i.e.,
\[
L_N = \{i \in [n] \mid i \text{ is distance } N \text{ from } 1\}.
\]
Since $G$ is connected, we have a disjoint union
\begin{equation}
\label{layerunion}
[n] = L_0 \cup L_1 \cup \cdots \cup L_h,
\end{equation}
where $h = \max\{N \mid L_N \ne \emptyset\}$.  Here is a simple
property of layers.

\begin{lemma}[{[1, Lem.\ 2.6]}]
\label{layerlem}
Let $G$ be labeled and connected. If $i \in L_N$ and $\{i,j\} \in
E(G)$, then $j \in L_{N-1}$, $L_N$, or $L_{N+1}$.
\end{lemma}

When $G$ is closed and connected, the layers are especially nice.

\begin{proposition}[{[1, Prop.\ 2.7]}]
\label{layerprop}
If $G$ is connected with a closed labeling, then:
\begin{enumerate}
\item Each layer $L_N$ is complete.
\item If $d = \max\{L_N\}$, then $L_{N+1} = \Linkp(d)$.
\end{enumerate}
\end{proposition}

The diameter of $G$ is denoted $\Diam(G)$, and a \emph{longest
shortest path} of $G$ is a shortest path of length $\Diam(G)$.
These concepts relate to layers as follows.

\begin{proposition}[{[1, Prop.\ 2.8]}]
\label{diameters}
If $G$ is connected with a closed labeling, then:
\begin{enumerate}
\item $\Diam(G)$ is the integer $h$ appearing in \eqref{layerunion}.
\item If $P$ is a longest shortest path of $G$, then one endpoint of
  $P$ is in $L_0$ or $L_1$ and the other is in $L_h$, where
  $h=\Diam(G)$.
\end{enumerate}
\end{proposition}

\section{Exchangeable Vertices} 
\label{exchangeable}

A closed graph with at least two vertices has at least two closed
labelings, since the reversal of a closed labeling is clearly closed.
But there may be other closed labelings, as shown by the simple
example:
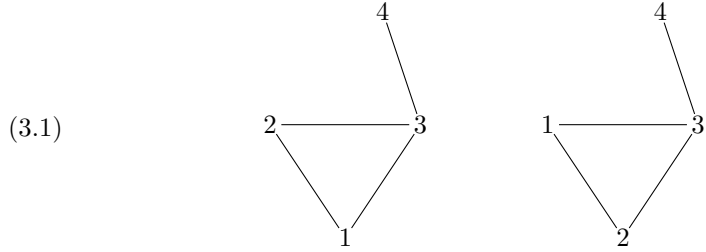
\begin{equation}
\label{simpleex}
\begin{array}{ccc}
\begin{tikzpicture}
  \node[vertex] (n1)  at (2,1) {1};
  \node[vertex] (n2)  at (1,2.5)  {2};
  \node[vertex] (n3) at (3,2.5) {3};
 
  \node[vertex] (n4) at (2.5,4){4};

  \foreach \from/\to in {n1/n2,n1/n3,n2/n3, n3/n4}
\draw (\from)--(\to);;

\end{tikzpicture}&\hspace{20pt}&
\begin{tikzpicture}
  \node[vertex] (n1)  at (2,1) {2};
  \node[vertex] (n2)  at (1,2.5)  {1};
  \node[vertex] (n3) at (3,2.5) {3};
 
  \node[vertex] (n4) at (2.5,4){4};

  \foreach \from/\to in {n1/n2,n1/n3,n2/n3, n3/n4}
\draw (\from)--(\to);;

\end{tikzpicture}
\end{array}
\end{equation}
To explore what makes this example work, we need some definitions.

\begin{definition}
\label{nbdplus}
Let $G$ be a graph.
\begin{enumerate}
\item The \emph{full neighborhood} of a vertex $v \in V(G)$ is
$\Linki{v} = \{v\} \cup \Link(v)$.
\item $v,w \in V(G)$ are \emph{exchangeable}, written $v \sim w$, if
  $\Linki{v} = \Linki{w}$.
\end{enumerate}
\end{definition}

Vertices $1$ and $2$ are exchangeable in the left-hand graph of
\eqref{simpleex}.  Switching labels gives the right-hand graph, which is
still closed.  Here is the general result.

\begin{proposition}
\label{prop:labeling1}
Let $G$ have a closed labeling.  If $i,j \in [n]$, $i \ne j$, are
exchangeable, then the labeling that switches $i$ and $j$ is also
closed.
\end{proposition}

\begin{proof}
Define $\phi : [n] \to [n]$ by $\phi(i) = j$, $\phi(j) = i$, and
$\phi(k) = k$ for $k \in [n] \setminus \{i,j\}$.  Pick $u,v,w\in
V(G)$ with $\{u,v\},\{v,w\}\in E(G)$, $u\ne w$, and
$\phi(u)>\phi(v)<\phi(w)$ or $\phi(u)<\phi(v)>\phi(w)$.  We need to
prove that $\{u,w\}\in E(G)$.

If $\{i,j\}\cap\{u,v,w\} = \emptyset$, then $\{u,w\}\in E(G)$ since
the original labeling is closed.  Now suppose $\{i,j\}\cap\{u,v,w\}
\ne \emptyset$ and $\phi(u)>\phi(v)<\phi(w)$.  There are several cases
to consider.  First suppose that $i=v$.  
If $j \in \{u,w\}$, then
without loss of generality we may assume $j = u$.  Then
\[
w \in \Linki{v} = \Linki{i} = \Linki{j} = \Linki{u}
\]
implies $\{u,w\} \in E(G)$.  
If $j \notin \{u,w\}$, then
$\phi(u)>\phi(i)<\phi(w)$ means that $u > j < w$.  Then $\{u,w\} \in
E(G)$ since the original labeling is closed and $j \sim i = v$.

The proof when $j=v$ is similar and is omitted.  Then two cases remain:
\begin{itemize}
\item $i = u$ and $j \notin \{v,w\}$. 
Thus $\phi(u) > \phi(v) <
  \phi(w)$ means that $j>v<w$.  Then $\{j,w\} \in E(G)$ since the
  original labeling is closed and $j \sim i = u$.  Using $j \sim i =
  u$ again, we conclude that
  $\{u,w\}\in E(G)$.
\item $i = u$ and $j = w$.  Then $\phi(u)>\phi(v)<\phi(w)$ means
  $j>v<i$. Then $\{u,w\} = \{i,j\} \in E(G)$ since the original
  labeling is closed. 
\end{itemize}

The proof when $\phi(u)<\phi(v)>\phi(w)$ is similar and is omitted. 
\end{proof}

Exchangeability $v\sim w$ is an equivalence relation on
$V(G)$ with equivalence classes
\[
e(v) = \{w \in V(G) \mid w \sim v\} = \{w \in V(G) \mid \Linki{w} =
\Linki{v}\}.
\]  
Equivalence classes are complete, since $v \sim w$ implies $v \in
\Linki{v} = \Linki{w}$, so that $\{v,w\} \in E(G)$ whenever $v \ne w$.

Since permutations are generated by transpositions,
Proposition~\ref{prop:labeling1} implies that when $G$ has a closed
labeling, every permutation of an equivalence class yields a new
closed labeling.

When $G$ is connected and closed, equivalence classes have
the following structure.

\begin{proposition}
\label{interval}
If $G$ is connected with a closed labeling and $i \in [n]$, then the
equivalence class $e(i)$ is an interval.
\end{proposition}

\begin{proof}
It suffices to show that if $i < j$ are exchangeable and $i < k < j$,
then $\Linki{k}= \Linki{i}$.  First note that $\{i,k\} \in E(G)$ since
$j \in \Linkp(i)$ and $\Linkp(i)$ is an interval by
Proposition~\ref{nbdinterval}.  Then $\{j,k\} \in E(G)$ since $i\sim
j$.

Now take $m \in \Linki{k}$.  We need to show $m \in \Linki{i}$.  If $m
= k$, this follows from the previous paragraph.  If $\{m,k\} \in
E(G)$, there are two possibilities:
\begin{itemize}
\item If $m < k$, then $m < k > i$, so $\{m,i\} \in
E(G)$ since the labeling is closed.
\item If $m > k$, then $m > k < j$, so either $m=j$ or $\{m,j\} \in
  E(G)$ by closed. 
\end{itemize}
Since $\Linki{i} = \Linki{j}$, both possibilities imply $m \in
\Linki{i}$.

Conversely, take $m \in \Linki{i}$.  If $m = i$, then $m \in \Linki{k}$
since $\{i,k\} \in E(G)$ by the first paragraph of the proof.  If
$\{m,i\} \in E(G)$, then $\{m,j\} \in E(G)$ since $i\sim j$.  Again,
there are two possibilities: 
\begin{itemize}
\item If $m < i$, then $m < i < k < j$, so $\{m,k\} \in
E(G)$ since $\Linkp(m)$ is an interval.
\item If $m > i$, then $m > i < k$, so either $m=k$ or $\{m,k\} \in
E(G)$ by closed.
\end{itemize}
Thus $m \in \Linki{k}$ and the proof is complete.
\end{proof}

\section{Counting Closed Labelings}
\label{collapsed}

Some graphs have no nontrivial exchangeable vertices.

\begin{definition}
\label{collapseddef}
A graph $G$ is \emph{collapsed} if all exchangeable vertices are
equal, i.e., $\Linki{v} = \Linki{w}$ implies $v = w$.
\end{definition}

\begin{proposition}
\label{collapsedthm}
Let $G$ be a closed graph with at least three vertices.  Then the following
are equivalent:
\begin{enumerate}
\item $G$ has exactly two closed labelings.  
\item $G$ is connected and collapsed.
\end{enumerate}
\end{proposition}

\begin{proof}
The proof of (1) $\Rightarrow$ (2) is easy.  If $G$ is not connected,
then $G$ is a disjoint union $G = G_1 \cup G_2$, where $G_i$ is
closed.  We may assume $G_1$ has at least two vertices, so $G_1$ has
at least two labelings.  Then we get at least four closed labelings of
$G$:\ two where $1$ is in $G_1$, and two where $1$ is in $G_2$.  Also,
if $G$ is not collapsed, then some equivalence class $e(i)$ has at
least two elements.  If $|e(i)| \ge 3$, then switching labels within
$e(i)$ gives at least $6$ closed labeling, and if $|e(i)| = 2$, then
$G$ has at least one more vertex, which makes it easy to see $G$ has
at least four closed labelings.

The proof of (2) $\Rightarrow$ (1) will take more work.  First note
that $G$ has diameter $\Diam(G) = h \ge 2$.  This follows because $h =
1$ would imply that $G$ is complete, which is impossible since $G$ is
collapsed with $\ge 3$ vertices, and $h = 0$ is impossible since $G$
is connected with $\ge 3$ vertices.

Fix a closed labeling with $V(G) = [n]$.  This gives layers $L_0 =
\{1\}, L_1, \dots, L_h$ associated with the labeling, and
Proposition~\ref{diameters}(2) implies that every longest shortest path
has one endpoint in $L_0$ or $L_1$ and the other in $L_h$.

Let $\phi:[n] \to [n]$ be another closed labeling which we will call
the $\phi$-labeling.  Pick $1' \in [n]$ such that $\phi(1') = 1$.  Then
some longest shortest path of $G$ begins at $1'$.  By the previous
paragraph, $1' \in L_0 \cup L_1$ or $1' \in L_h$.  Replacing $\phi$
with its reversal if necessary, we may assume that $1' \in L_0 \cup
L_1$.  We claim that $\phi$ is the identity function.  This will prove
the theorem.

We first show that $1' = 1$, i.e, $\phi(1) = 1$.  Recall that $L_1=
\Link(1)$ and that $L_1$ is complete by Propostion~\ref{layerprop}(1).
It follows that $\Linki{1} = L_0 \cup L_1$ is also complete.  The same
argument implies that $\Linki{1'}$ is complete.  Now suppose $1 \ne
1'$ and pick $m \in \Linki{1'}$ is different from $1$.  Then $\{1,m\}
\in E(G)$ since $1 \in \Linki{1'}$ and $\Linki{1'}$ is complete.  This
implies $m \in L_1 = \Link(1)$, and then the inclusion $\Linki{1'}
\subseteq \Linki{1}$ follows easily.  The opposite inclusion follows
by interchanging the two labelings.  Hence we have proved $\Linki{1'}
= \Linki{1}$.  Since we are assuming $1 \ne 1'$, this contradicts
the fact that $G$ is collapsed.  Hence we must have $1' = 1$, as claimed.

Now suppose that vertices $1,\dots,u-1\in [n]$ have the same
$\phi$-label as in the original labeling, i.e., $\phi(j) = j$ for $1
\le j \le u-1$.  Then pick $u' \in [n]$ such that $\phi(u') = u$.  To
prove that $u' = u$, i.e., $\phi(u) = u$, suppose that $u' \ne u$.
Since $\phi$ is the identity on $1,\dots,u-1$ and $\phi(u') = u$, we
have $u' > u$ and $\phi(u') < \phi(u)$.

We first show that $\{u,u'\} \in E(G)$.  Since $G$ is
connected, Proposition~\ref{nbdinterval} implies that every vertex is
connected by an edge to its successor in any closed labeling.  For the
original labeling, this gives $\{u-1,u\} \in E(G)$, and for the
$\phi$-labeling, this gives $\{u-1,u'\} \in E(G)$ since $\phi(u-1) =
u-1$ and $\phi(u') = u$.  Proposition~\ref{nbdinterval} implies that
$\Linkp(u-1)$ (in the original labeling) is complete, and $\{u,u'\}
\in E(G)$ follows.

We next prove that $\Linki{u}\subseteq\Linki{u'}$.  Pick $m\in
\Linki{u}$.  Then:
\begin{itemize}
\item If $m = u$, then $m \in \Linki{u'}$ since $\{u,u'\} \in E(G)$.  
\item If $m > u$, then either $m = u'$, in which case $m \in
  \Linki{u'}$ is obvious, or $m \ne u'$, in which case $m \in
  \Linki{u'}$ since $m > u < u'$ implies $\{m,u'\} \in E(G)$ as the 
  original labeling is closed.
\item If $m < u$, then $m \in \Linki{u'}$ since $\phi(m) = m < u <
  \phi(u) > \phi(u')$ implies $\{m,u'\} \in E(G)$ as the
  $\phi$-labeling is closed.
\end{itemize}
This proves $\Linki{u} \subseteq \Linki{u'}$.  By symmetry, we get
$\Linki{u'}=\Linki{u}$, which contradicts $u' \ne u$ since $G$ is
collapsed. We conclude that $u'=u$, and then $\phi$ is the identity by
induction on $u$.  This completes the proof.
\end{proof}

Now suppose that $G$ is a connected graph with a closed labeling.
Since each equivalence class is an interval by
Proposition~\ref{interval}, we can order the equivalence classes
\begin{equation}
\label{orderE}
E_1 < E_2 < \cdots < E_r
\end{equation}
so that if $i \in E_a$ and $j \in E_b$, then $i < j$ if and only if $a
< b$.  This induces an ordering on $V(G)/\!\!\sim \ =
\{E_1,\dots,E_r\}$.  Then define the graph $G/\!\!\sim$ with vertices
\begin{equation}
\label{Gcvert}
V(G/\!\!\sim) = V(G)/\!\!\sim\ = \{E_1,\dots,E_r\}
\end{equation}
and edges
\begin{equation}
\label{Gcedge}
E(G/\!\!\sim) = \big\{\{E_a,E_b\} \mid \{i,j\} \in E(G) \text{ for
  some } i\in E_a, j \in E_b\big\}.
\end{equation}
Since $i\sim i'$ and $j \sim j'$ imply that $\{i,j\} \in E(G)$ if and
only if $\{i',j'\} \in E(G)$, we can replace ``for some'' with ``for
all'' in \eqref{Gcedge}.

\begin{theorem}
\label{quotient}
Let $G$ be connected with a closed labeling and exchangeable
equivalence classes $E_1,\dots,E_r$.  Then: 
\begin{enumerate}
\item The quotient graph $G/\!\!\sim$ defined in \eqref{Gcvert} and
  \eqref{Gcedge} is connected, collapsed, and closed with respect to
  the labeling \eqref{orderE}.
\item If $r > 1$, then $G$ has precisely $2\prod_{a=1}^r |E_a|!$
  closed labelings.
\end{enumerate}
\end{theorem}

\begin{proof}
For (1), we omit the straightforward proof that $G/\!\!\sim$ is
connected and closed with respect to \eqref{orderE}.  To prove that
$G/\!\!\sim$ is collapsed, we first observe that for vertices $u,v \in
V(G)$,
\begin{equation}
\label{eqclassiff}
u \in \Linki{v} \iff e(u) \in \Linki{e(v)}.
\end{equation}
We leave the simple proof to the reader.  Now suppose that equivalance
classes $e(v), e(w)$ satisfy $e(v) \sim e(w)$.  Then by
\eqref{eqclassiff}, we have
\[
u \in \Linki{v} \Leftrightarrow e(u) \in \Linki{e(v)} \Leftrightarrow
e(u) \in \Linki{e(w)} \Leftrightarrow u \in \Linki{w}. 
\]
This proves that $\Linki{v} = \Linki{w}$.  Then $v \sim w$, which
implies $e(v) = e(w)$.  It follows that $G/\!\!\sim$ is collapsed. 

For (2), first note that $r > 1$ implies $r \ge 3$, for if there were
only two equivalance classes $E_1$ and $E_2$, then since $G$ is
connected there must be $\{v,w\} \in E(G)$ with $v \in E_1$ and $w \in
E_2$.  The observation following \eqref{Gcedge} implies that $\{s,t\}
\in E(G)$ for all $s \in E_1$ and $t \in E_2$.  It follows easily that
$G$ is complete, which implies $r = 1$, a contradiction.  Hence $r \ge
3$.
 
According to Proposition~\ref{collapsedthm}, $G/\!\!\sim$ has exactly
two closed labelings since it has $r \ge 3$ vertices by the previous
paragraph and is connected, closed, and collapsed by (1).  It follows
from \eqref{orderE} that any closed labeling of $G$ induces one of
these two closed labelings of $G/\!\!\sim$.  Hence all closed
labelings of $G$ arise from the two ways of ordering the equivalance
classes, together with how we order elements within each equivalance
class.  Proposition~\ref{prop:labeling1} and the remarks following the
proposition imply that we can use any of the $|E|!$ orderings of the
elements of an equivalance class $E$.  Since different equivalence
classes can be ordered independently of each other, we get the desired
formula for the total number of closed orderings of $G$.
\end{proof}

\section{Counting Closed Graphs} 
\label{counting}

In Theorem~\ref{quotient}, we fixed a connected graph and counted
the number of closed labelings.  Here we change the point of view,
where we fix a labeling and count the the number of connected graphs
for which the given labeling is closed.

Here is how a layer of a connected closed graph
connects to the next layer.

\begin{definition}
\label{countingdefs}
Let $G$ be a connected graph with a closed labeling.  Let the layers
of $G$ be $L_0 = \{1\}, L_1, \dots, L_h$, $h = \Diam(G)$.
\begin{enumerate}
\item Let $a_N = |L_N|$ for $N = 0,\dots,h$.  Note that $a_0 = 1$.
\item If $N < h$, write the vertices of $L_N$ in order.  For $1 \le s
  \le a_N$, let $b_s$ be the number of edges of $G$ connecting the
  $s^{\mathit{th}}$ vertex of $L_N$ to a vertex of $L_{N+1}$.
\item The \emph{sequence} of $L_N$ is the sequence $S_N=(b_1, b_2,
  \dots, b_{a_N})$.
\end{enumerate}
\end{definition}

Here is some further notation we will need.  First, let $m_N =
\min\{L_N\}$ be the minimal element of the $L_N$.
Propositions~\ref{nbdinterval} and~\ref{layerprop} imply that $L_N$ is
complete and is an interval.  Thus $L_N = [m_N, m_N+a_N-1]$, and the
$s^{\mathit{th}}$ vertex of $L_N$ is $u_s = m_N + s-1$.

We can now show that the sequence $S_N = (b_1,b_2,\dots, b_{a_N})$
determines precisely how $L_N$ is connected to $L_{N+1}$.

\begin{proposition}
\label{prop:uniqueSequences}
Let $G$ be connected with a closed labeling.
If $u_s = m_N + s-1 \in L_N$ is the $s^{\mathit{th}}$ vertex of $L_N$
and $b_s > 0$, then
\[
\{v \in L_{N+1} \mid \{u_s,v\} \in E(G)\} = [m_{N+1},m_{N+1}+b_s-1].
\]
Thus $b_s$ determines how $u_s$ links to $L_{N+1}$.
\end{proposition}

\begin{proof}
Let $A = \{v \in L_{N+1} \mid \{u_s,v\} \in E(G)\}$.  Note that every
$v \in A$ satisfies $v > u_s$ by Proposition~\ref{layerprop}(2).  It
follows easily that
\[
A = \Linkp(u_s) \cap L_{N+1}.
\]
We know that $L_{N+1}$ is an interval, and the same is true for
$\Linkp(u_s)$ by Proposition~\ref{nbdinterval}.  Hence $A$ is an
interval.  However, if $v \in A$ and $v \ne m_{N+1}$, then $m_{N+1} <
v > u_s$ and closed imply $\{u_s,m_{N+1}\} \in E(G)$ since
$\{m_{N+1},v\} \in E(G)$ by the completeness of $L_{N+1}$.  Hence
$m_{N+1} \in A$, and from here, the proposition follows without
difficulty.
\end{proof}

Here is an important property of the sequence $S_N$.

\begin{proposition}
\label{prop:increaseSequences}
Let $G$ be connected with a closed labeling.  If $N <
\Diam(G)$, then the sequence $S_N = (b_1,b_2,\dots, b_{a_N})$ of the
layer $L_N$ has the following properties:
\begin{enumerate}
\item The last element of $S_N$ is $a_{N+1}$, i.e., $b_{a_N} =
  a_{N+1}$.
\item $S_N$ is increasing, i.e., $b_s \le b_{s+1}$ for $s =
  1,\dots,a_N-1$.
\end{enumerate}
\end{proposition}

\begin{proof}
For (1), note that the last vertex of $L_N$ connects to every vertex
of $L_{N+1}$ by Proposition~\ref{layerprop}(2).  It follows that
$b_{a_N} = |L_{N+1}| = a_{N+1}$.

For (2), let $u_s$ be the $s^{\mathit{th}}$ vertex of $L_N$, $1 \le s
\le a_N-1$.  If $b_s = 0$, then $b_s \le b_{s+1}$ clearly holds.  If
$b_s > 0$, then $u_s$ connects to $m_{N+1}+b_s-1$ by
Proposition~\ref{prop:uniqueSequences}, and it connects to $u_{s+1}$
since $L_N$ is complete.  Then $m_{N+1}+b_s-1 > u_s < u_{s+1}$ implies
that $u_{s+1}$ connects to $m_{N+1}+b_s-1$ since the labeling is
closed.  Using Proposition~\ref{prop:uniqueSequences} again, we obtain
\[
m_{N+1}+b_s-1 \in [m_{N+1},m_{N+1}+b_{s+1}-1],
\]
and $b_s \le b_{s+1}$ follows.
\end{proof}

We now come to the main result of this section.

\begin{theorem}
\label{partition}
Fix $n$ and an integer partition $n = a_0 + a_1 + \dots + a_h$ with
$a_0 = 1$ and $a_N \ge 1$ for $N = 1,\dots,h$.  Also set
$\mathcal{L}_0 = \{1\}$ and 
\begin{equation}
\label{LNdesc}
\mathcal{L}_N = [a_0+\cdots+a_{N-1}+1,a_0+\cdots+a_N]
\end{equation}
for $N = 1,\dots,h$, so that $|\mathcal{L}_N| = a_N$.  Then the number
of graphs $G$ satisfying the conditions:
\begin{enumerate}
\item $V(G) = [n]$,
\item $G$ is connected and closed with respect to the labeling $V(G) =
  [n]$, and
\item The $N^{\mathit{th}}$ layer of $G$ is $\mathcal{L}_N$ for $N =
  0,\dots,h$,
\end{enumerate}
is given by the product
\[
\prod\limits_{N=0}^{h-1} \binom{a_{N+1}+a_N-1}{a_N-1}.
\]
\end{theorem}

\begin{proof}
Let $G$ satisfy (1), (2) and (3).  Each layer of $G$ is complete, and
every edge of $G$ connects to the same layer or an adjacent layer by
Lemma~\ref{layerlem}.  Then Proposition~\ref{prop:uniqueSequences}
shows that the edges of $G$ are uniquely determined by
$S_0,\dots,S_{h-1}$.

By Proposition~\ref{prop:increaseSequences}, each $S_N =
(b_1,b_2,\dots, b_{a_N})$ is an increasing sequence of nonnegative
integers of length $a_N$ that ends at $a_{N+1}$.  It is well known
that the number of such sequences equals the binomial coefficient
$\binom{a_{N+1}+a_N-1}{a_N-1}$.  

It follows that the product in the statement of the proposition is an
upper bound for the number of graphs satisfying (1), (2) and (3).

To complete the proof, we need to show that every sequence counted by
the product corresponds to a graph $G$ satisfying (1), (2) and (3).  
First note that the minimal element of $\mathcal{L}_N$ is
\[
m_N = a_0+\cdots+a_{N-1}+1
\]
when $N > 0$.  Now suppose we have sequences $S_0,\dots,S_{h-1}$,
where each $S_N = (b_1,b_2,\ldots, b_{a_N})$ is an increasing sequence
of nonnegative integers of length $a_N$ that ends at $a_{N+1}$.  This
determines a graph $G$ with $V(G) = [n]$ and the following edges:
\begin{itemize}
\item[(A)] All possible edges connecting elements in the same level
$\mathcal{L}_N$.
\item[(B)] For each $N = 0,\dots,h-1$, all edges $\{u_s,v\}$, where $u_s$
  is the $s^{\mathit{th}}$ vertex of $\mathcal{L}_N$ and $v$ is any
  vertex in the interval $[m_{N+1},m_{N+1}+b_s-1] \subseteq
  \mathcal{L}_{N+1}$ from 
  Proposition~\ref{prop:uniqueSequences}. 
\end{itemize}
Once we prove that $G$ is closed and connected with $\mathcal{L}_N$ as
its $N^{\mathit{th}}$ layer, the theorem will be proved.

Since $b_{a_N} = a_{N+1}$, we see that for $N = 0,\dots,h-1$, the last
elment of $\mathcal{L_N}$ connects to all elements of
$\mathcal{L}_{N+1}$.  This enables us to construct a path from $1$ to
any $u \in \mathcal{L}_N$ for $N = 1,\dots,h$.  It follows that $G$ is
connected and that all $u \in \mathcal{L}_N$ have distance at most $N$
from vertex $1$.  Since every edge of $G$ connects elements of
$\mathcal{L}_M$ to $\mathcal{L}_M$, $\mathcal{L}_{M+1}$, or
$\mathcal{L}_{M-1}$, any path connecting $1$ to $u\in \mathcal{L}_N$
must have length at least $N$.  It follows that $\mathcal{L}_N$ is
indeed the $N^{\mathit{th}}$ layer of $G$.

It remains to show that $G$ is closed with respect to the natural
labeling given by $V(G) = [n]$.  A vertex of $G$ is the
$s^{\mathit{th}}$ vertex $u_s$ of $\mathcal{L}_N$ for some $s$ and
$N$.  We will show that $\Linkp(u_s)$ satisfies
Proposition~\ref{nbdinterval}.  The formula \eqref{LNdesc} for
$\mathcal{L}_N$ and the description of the edges of $G$ given in (A)
and (B) make it clear that
\begin{align*}
\Linkp(u_s) &= [u_{s+1},a_0+\cdots+a_N] \cup
      [m_{N+1},m_{N+1}+b_s-1]\\ &= [u_{s+1},m_{N+1}+b_s-1], 
\end{align*}
where the second equality follows from $m_{N+1} = a_0+\cdots +a_N+1$.
To show that $\Linkp(u_s)$ is complete, take distinct vertices $v,w
\in \Linkp(u_s)$.  If both lie in $\mathcal{L}_N$ or
$\mathcal{L}_{N+1}$, then $\{v,w\} \in V(G)$ by (A).  Otherwise, we
may assume without loss of generality that $v = u_t$, $t \ge s$, and
$w \in [m_{N+1},m_{N+1}+b_s-1]$.  Note that $u_t$ links to every
vertex in $[m_{N+1},m_{N+1}+b_t-1]$ by (B).  We also have $b_s \le
b_t$ since $S_N$ is increasing.  It follows that $\{v,w\} = \{u_t,w\}
\in E(G)$.  Hence $\Linkp(u_s)$ is complete, so that $G$ is closed by
Proposition~\ref{nbdinterval}.
\end{proof}

\section{Local Clustering Coefficients} 
\label{clustering}

In a social network, one
can ask how often a friend of a friend is also a friend.  Translated
into graph theory, this asks how often a path of length two has an
edge connecting the endpoints of the path.  The illustration
\eqref{closedpicture} from the Introduction indicates that this should
be a frequent occurrence in a closed graph.   

There are several ways to quanitify the ``friend of a friend''
phenomenon.  For our purposes, the most convenient is the \emph{local
  clustering coefficient} of vertex $v$ of a graph $G$, which is
defined by
\[
C_v = \begin{cases}{\displaystyle\frac{\text{number of pairs of
        neighbors of $v$ connected by an edge}}{\text{number of pairs
        of neighbors of $v$}}} & \deg(v) \ge 2\\ 0 & \deg(v) \le
  1.\end{cases} 
\]
Local clustering coefficients are discussed in \cite[pp.\ 201--204]{N}.

\begin{proposition}
\label{cluster}
Let $v$ be a vertex of a closed graph $G$ of degree $d = \deg(v) \ge
2$.  Then the local clustering coefficient $C_v$ satisfies the
inequality
\[
C_v \ge \frac12 - \frac1{2(d-1)}.
\]
Furthermore, $d \ge 3$ implies that $C_v \ge \frac13$.
\end{proposition}

\begin{proof}
Pick a closed labeling of $G$ and let $a = |\Linkp(v)|$ and $b =
|\Linkm(v)|$.  Then $a+b = |\Link(v)| = \deg(v) = d$.   Since the
labeling is closed, any pair of vertices in $\Linkp(v)$ or in
$\Linkm(v)$ is connected by an edge.  It follows that at least 
\[
{\textstyle \frac12} a(a-1) + {\textstyle \frac12} b(b-1) 
\]
pairs of neighbors of $v$ are connected by an edge.  Since the total
number of such pairs is $\frac12 d(d-1)$ and $d = a+b$, we obtain 
\begin{equation}
\label{Cvineq}
C_v \ge \frac{a(a-1) + b(b-1)}{d(d-1)} = \frac{a^2+b^2-d}{d(d-1)} \ge
\frac{\ {\textstyle\frac12} d^2 - d\ }{d(d-1)} = \frac12 -
\frac1{2(d-1)}, 
\end{equation}
where we use $a^2+b^2-\frac12 d^2 =  \frac12(a-b)^2 \ge 0$.  When $d
\ge 4$, this inequality for $C_v$ easily gives $C_v \ge \frac13$.
When $d = 3$, then $a+b = 3$, $a,b \in \Z$, implies that $a^2+b^2 \ge
5$, in which case the left half of \eqref{Cvineq} gives $C_v \ge
\frac{5-3}{3(3-1)} = \frac13$.  
\end{proof}

A global version of the clustering coefficient defined by Watts and
Strogatz is
\[
C_{\mathrm{WS}} = \frac1n \sum_{v \in V(G)} C_v, \quad n = |V(G)|.
\]
(See reference [323] of \cite{N}.  A different global clustering
coefficient is discussed in \cite[pp.\ 199--204]{N}.)   To estimate
$C_{\mathrm{WS}}$ for a closed graph, we need the following lemma.  

\begin{lemma}
\label{lemma:diameterAndCluster}
Let $G$ be a connected closed graph.
\begin{enumerate}
\item Set $h=\Diam(G)$ and let $c$ be the number of vertices $v \in G$
  with $\deg(v)=2$ and $C_v=0$. Then $c\leq h-1$.
\item $G$ has at most two leaves.
\end{enumerate}
\end{lemma}

\begin{proof}
For (1), fix a closed labeling for $G$ with $V(G) = [n]$ and pick
$v\in V(G)$ with $\deg(v) = 2$ and $C_v = 0$.  We claim that $v$ is in
a layer of its own.  To see why, let $v \in L_N$ and suppose there is
$s\in L_N$ with $s\neq v$.  Then $\{v,s\}\in E(G)$ since layers are
complete by Proposition~\ref{layerprop}(1).  Furthermore, $|L_N| \ge
2$, so $N > 0$.  Then $\{s,d\}$, $\{v,d\} \in E(G)$ for $d =
\max\{L_{N-1}\}$ by Proposition~\ref{layerprop}(2).  Since $\deg(v) =
2$, we must have $\Link(v) = \{s,d\}$, and then $\{s,d\} \in E(G)$
contradicts $C_v = 0$.  Thus $\{v\}$ is a layer when $\deg(v) = 2$ and
$C_v = 0$.

Note that if $\{v\} = L_0$, then the two vertices in $\Link(v) = L_1$
would be linked by an edge.  The same holds if $\{v\} = L_h$, for here
the two vertices would be in $L_{h-1}$ since $L_h$ is the highest
layer by Proposition~\ref{diameters}(1). It follows that each of the $c$
vertices with $\deg(v)=2$ and $C_v=0$ lies in a separate layer
distinct from $L_0$ or $L_h$.  Since there are only $h-1$ intermediate
layers, we must have $c\leq h-1$.  

For (2), assume $G$ has leaves $u,v,w$ and fix a
closed labeling of $G$.  We may assume $u < v < w$, and let $u',v',w'$
be the unique vertices adjacent to $u,v,w$ respectively.  A shortest
path from $u$ to $v$ is directed (see \cite{H} or Proposition 2.1 of
\cite{CE}) and must pass through $u'$ and $v'$, hence $u < u' \le v' <
v$ since $u < v$.  The same argument applied to $v$ and $w$ would
imply $v < v' \le w' < w$.  Thus $v' < v$ and $v < v'$, so three
leaves cannot exist. 
\end{proof}

We can now estimate the clustering coefficient $C_{\mathrm{WS}}$ of a
closed graph.

\begin{theorem}
\label{CWSprop}
If $G$ is connected and closed with $n > 1$ vertices and
diameter~$h$, then  
\[
C_{\mathrm{WS}} \geq \frac{1}{3}-\frac{h+1}{3n}.
\]
\end{theorem}

\begin{proof}
Since $n > 1$ and $G$ is connected, all vertices of $G$ have degree
$\ge 1$.  Thus we can write $V(G)$ as the disjoint union 
\[
V(G) = \mathcal{A} \cup \mathcal{B} \cup \mathcal{C} \cup \mathcal{D},
\]
where $\mathcal{A}$ consists of vertices of degree $\ge3$,
$\mathcal{B}$ consists of vertices of degree $2$ with $C_v = 1$,
$\mathcal{C}$ consists of vertices of degree $2$ with $C_v = 0$, and
$\mathcal{D}$ consists of the leaves (which have $C_v = 0$).   Since
$C_v \ge \frac13$ for $v \in \mathcal{A}$ by
Proposition~\ref{cluster}, we have 
\[
C_{\mathrm{WS}} \ge \frac1n\Big( \frac13\cdot|\mathcal{A}| + 1\cdot
|\mathcal{B}| + 0\cdot |\mathcal{C}| + 0\cdot  |\mathcal{D}| \Big) \ge
\frac{|\mathcal{A}| +  |\mathcal{B}|}{3n}\ 
= \frac{n - (|\mathcal{C}| +  |\mathcal{D}|)}{3n}.
\]
Then we are done since $|\mathcal{C}| \leq h-1$ and $|\mathcal{D}| \le
2$ by Lemma~\ref{lemma:diameterAndCluster}.
\end{proof}

By Theorem~\ref{CWSprop}, the clustering coefficient
$C_{\mathrm{WS}}$ is large when the diameter is small compared to the
number of vertices.  At the other extreme, both sides of the
inequality in Proposition~\ref{CWSprop} are zero when $G$ is a path
graph.

\section*{Acknowledgements}

We are grateful to Amherst College for the Post-Baccalaureate Summer
Research Fellowship that supported the writing of this paper.  Thanks
also to Amy Wagaman for suggesting that we look at clustering
coefficients.

\end{document}